\documentclass[11pt,a4paper]{amsart}
\usepackage{amscd}

\usepackage{amsmath}
\usepackage{amsthm}
\usepackage{amssymb}

\usepackage[T1]{fontenc}
\usepackage[english]{babel}
\usepackage[utf8]{inputenc}
\usepackage{stmaryrd}
\usepackage{rotating}
\usepackage[all,poly,necula]{xy}
\usepackage{ulem}

\usepackage{color}
\definecolor{orange}{rgb}{1,0.90,0.5}
\let\oldmarginpar\marginpar
\renewcommand\marginpar[1]{{\color{red}(**)}\oldmarginpar[\raggedleft\footnotesize \fcolorbox{blue}{orange}{\parbox{\marginparwidth}{\color{blue}{(**) #1}}}]%
{\raggedright\footnotesize \fcolorbox{blue}{orange}{\parbox{\marginparwidth}{\color{blue}{(**) #1}}}}}

\theoremstyle{plain}
\newtheorem{thm}{Theorem}[section]
\newtheorem*{thmw}{Theorem}

\newtheorem{cor}[thm]{Corollary}
\newtheorem*{corw}{Corollary}
\newtheorem{lem}[thm]{Lemma}
\newtheorem{prop}[thm]{Proposition}

\theoremstyle{definition}

\newtheorem{exm}[thm]{Example}

\theoremstyle{remark}
\newtheorem{rem}[thm]{Remark}

\def\s{\stackrel}

\renewcommand{\mod}{\operatorname{mod}\nolimits}
\newcommand{\Mod}{\operatorname{Mod}\nolimits}
\newcommand{\add}{\operatorname{add}\nolimits}
\newcommand{\diag}{\operatorname{diag}\nolimits}
\newcommand{\GL}{\operatorname{GL}\nolimits}

\newcommand{\D}{\operatorname{D}\nolimits}

\newcommand{\Hom}{\operatorname{Hom}\nolimits}

\newcommand{\uHom}{\operatorname{\underline{Hom}}\nolimits}
\newcommand{\uEnd}{\operatorname{\underline{End}}\nolimits}
\newcommand{\End}{\operatorname{End}\nolimits}
\renewcommand{\Im}{\operatorname{Im}\nolimits}

\newcommand{\Ext}{\operatorname{Ext}\nolimits}
\newcommand{\SL}{\operatorname{SL}\nolimits}

\newcommand{\op}{\operatorname{op}\nolimits}
\newcommand{\CM}{\operatorname{CM}\nolimits}

\newcommand{\M}{\mathcal M}
\newcommand{\uM}{\underline{\M}}
\newcommand{\oM}{\overline{\M}}
\newcommand{\B}{\mathcal B}
\newcommand{\uB}{\underline{\B}}
\newcommand{\oB}{\overline{\B}}
\newcommand{\C}{\mathcal C}
\newcommand{\uC}{\underline{\C}}
\newcommand{\oC}{\overline{\C}}
\newcommand{\Omegab}{\overline{\Omega}}
\newcommand{\CC}{\mathbb C}

\newcommand{\vecv}[2]{{\left(\!\!
    \begin{array}{c}
      #1 \\
      #2
    \end{array}
    \!\!\right)}}
\newcommand{\vech}[2]{{\left(\!\!
    \begin{array}{cc}
      #1 & #2
    \end{array}
    \!\!\right)}}

\newcommand{\svecv}[2]{\left(\begin{smallmatrix}
      #1 \\
      #2
    \end{smallmatrix}\right)}

\newcommand{\svech}[2]{\left(\begin{smallmatrix}
      #1 & #2
\end{smallmatrix}\right)}

\renewcommand{\phi}{\varphi}

\renewcommand{\emph}{\textit}

\newsavebox\locboxinminipage
\newlength\locboxinminipagel
\newcommand{\boxinminipage}[1]
{%
 \sbox\locboxinminipage{#1}%
 \settowidth\locboxinminipagel{\usebox{\locboxinminipage}}%
 \begin{minipage}{\locboxinminipagel}\usebox{\locboxinminipage}\end{minipage}%
}

\begin{document}

\title[Quotients of exact categories by cluster tilting subcategories]{Quotients of exact categories by cluster tilting subcategories as module categories}
\author{Laurent Demonet}
\thanks{The first author was supported during this research by grants P10723 and 22000723 of Japan Society for Promotion of Science.}
\address{Graduate School of Mathematics \\ Nagoya University \\ 464-8602 Nagoya, Japan}
\email{Laurent.Demonet@normalesup.org}
\author{Yu Liu}
\address{Graduate School of Mathematics \\ Nagoya University \\ 464-8602 Nagoya, Japan}
\email{d11005m@math.nagoya-u.ac.jp}

\begin{abstract}
We prove that some subquotient categories of exact categories are abelian. This generalizes a result by Koenig-Zhu in the case of (algebraic) triangulated categories. As a particular case, if an exact category $\B$ with enough projectives and injectives has a cluster tilting subcategory $\M$, then $\B/\M$ is abelian. More precisely, it is equivalent to the category of finitely presented modules over $\uM$.
\end{abstract}

\keywords{exact category, quotient, rigid, cluster tilting, abelian category, equivalence}

\maketitle

\section{Introduction}

Recently, cluster tilting theory (see for example \cite{BMRRT, GLS, 5}) permitted to construct abelian categories from some triangulated categories. The main aim of this article is to provide a framework generalizing this observation to exact categories. By Buan-Marsh-Reiten \cite[Theorem 2.2]{BMR} in cluster categories, by Keller-Reiten \cite[Proposition 2.1]{KR} in the $2$-Calabi-Yau case and then by Koenig-Zhu \cite[Theorem 3.3]{7} in the general case, one can pass from triangulated categories to abelian categories by factoring out any cluster tilting subcategory.

The main aim of this article is to prove the analogous result for exact categories: \emph{The quotient of an exact category with enough projectives by a cluster tilting subcategory is abelian}. This new result seems \emph{a priori} less surprising than the one in triangulated categories because these ones are intuitively further to abelian categories. Notice that, applied to Frobenius categories, it is equivalent to the result for \emph{algebraic} triangulated categories (\emph{i.e.}~stable categories of Frobenius categories, see \cite{6}). However, most triangulated categories appearing in representation theory turn out to be in fact \emph{algebraic}, while most exact categories are not Frobenius. In this respect, the case of exact categories can be seen as a generalization of the result concerning triangulated categories, as well as a more natural version. To support this idea, we give examples coming from singularity theory and from $n$-representation-finite algebras \cite{IO} in last section of this paper.

Let $\B$ be an exact category with enough projectives and injectives and $\M$ be a full \emph{rigid} subcategory of $\B$ (\emph{i.e.} $\Ext^1_\B(\M, \M) = 0$ where $\Ext^1_\B(\M, \M)$ is the essential image of the bifunctor $\Ext^1_\B$). We denote by $\uM$ (resp. $\oM$) the full subcategory of the stable category $\uB$ (resp. costable category $\oB$) with the same objects as $\M$. Let $\M_L$ (resp. $\M_R$) be the subcategory of objects $X$ which admit short exact sequences
$$0{\rightarrow}X\xrightarrow{d^0}M^0\xrightarrow{d^1} M^1{\rightarrow} 0 \quad \text{(resp. } 0{\rightarrow}M_1 \xrightarrow{d_1}M_0\xrightarrow{d_0} X{\rightarrow} 0 \text{)}$$
with $M^0, M^1, M_0, M_1 \in \M$. In this case, $d^0$ (resp. $d_0$) is a left (resp. right) $\M$-approximation of $X$. Finally, $\Omegab  \M$ is the class of objects $X\in {\B}$ such that there exists a short exact sequence
$$0\rightarrow M\rightarrow I\rightarrow X\rightarrow 0$$
where $M\in {\M}$, and $I$ is injective.

The main result of this paper (theorems \ref{2.5} and \ref{2.7}) is:
\begin{thmw}
 There are equivalences of categories:
 \begin{enumerate}
  \item between the subquotient category $\M_L/[\M]$ of $\B$ and the category of finitely presented modules $\mod\uM$ if $\M$ contains all projective objects of $\B$;
  \item between the subquotient category $\M_R/[\Omegab \M]$ of $\B$ and the category of finitely presented modules $\mod\oM$ if $\M$ contains all injective objects of $\B$.
 \end{enumerate}
\end{thmw}

This permits to get the following corollary (see Lemma \ref{2.2}):
\begin{corw}
 If $\M$ is contravariantly finite, then $\M_L/[\M]$ is abelian.
\end{corw}

For the sake of generality, note that we have to consider categories of modules over additive categories. For more details about these, see \cite{ASS, 3}.

The previous results concern categories $\M_L$ and $\M_R$ which have not good properties in general. From now on, we suppose that $\M$ is $n$-cluster tilting for some integer $n \geq 2$ (see \cite{IyamaHigher,5}). Thus, the properties of $\M_L$ and $\M_R$ becomes much clearer: in particular, \begin{align*}
 & \M_L={}^{\bot_{n-2}} \M = \{ X\in {\B} \,|\, \forall i \in \{1, \dots, n-2\}, \Ext_{\B}^i(X, {\M})=0 \} \\
\text{and} \quad & \M_R = \M^{\bot_{n-2}} = \{ X\in {\B} \,|\, \forall i \in \{1, \dots, n-2\}, \Ext_{\B}^i({\M}, X)=0 \},
\end{align*}
and therefore they are exact subcategories of $\B$ (see Remark \ref{Mpex}). In particular we get

\begin{corw}
 If $\M$ is $2$-cluster tilting then $\B / [\M] \simeq \mod \uM$ is abelian.
\end{corw}

As a particular case of this corollary, we obtain (see Example \ref{exCM})
\begin{corw}
 Let $S = \CC\llbracket X, Y, Z \rrbracket$ and $G$ be a finite subgroup of $\GL_3(\CC)$ without pseudo-reflections acting on $S$. We suppose that $R = S^G$ is an isolated singularity. Then, denoting the category of Cohen-Macaulay modules over $R$ by $\CM(R)$,
 $$\frac{\CM(R)}{[S]} \simeq \mod \left(\frac{S * G}{(e)}\right)^{\op}$$
 where $S*G$ is the skew-group algebra and the idempotent $e$ corresponds to the trivial representation of $G$.
\end{corw}

 For example, if $G = \langle \diag(\zeta, \zeta, \zeta) \rangle$ where $\zeta$ is a primitive $n$-th root of the unity, then
 $$\frac{\CM(R)}{[S]} \simeq \mod \CC Q / I $$
 where
 $$Q = \boxinminipage{\xymatrix@C=2cm{\bullet \ar@/^/[r]^{x_1} \ar[r]|{y_1} \ar@/_/[r]_{z_1} & \bullet \ar@/^/[r]^{x_2} \ar[r]|{y_2} \ar@/_/[r]_{z_2} & \bullet \ar@{..}[r] & \bullet \ar@/^/[r]^{x_{n-2}} \ar[r]|{y_{n-2}} \ar@/_/[r]_{z_{n-2}} & \bullet }}$$
 and $I$ is the ideal of $\CC Q$ generated by the relations $x_{i} y_{i+1} = y_{i} x_{i+1}$, $y_{i} z_{i+1} = z_{i} y_{i+1}$, $z_{i} x_{i+1} = x_{i} z_{i+1}$ for $1 \leq i \leq n-3$.

Moreover, if $\B$ has an $(n-1)$-AR translation $\tau_{n-1}$ as defined in \cite{IyamaHigher}, we can combine the two equivalences in the following commutative diagram of equivalences (see Theorem \ref{2.14}):
$$\xymatrix{
{}^{\bot_{n-2}} \M/[\M] \ar[d]_{\tau_{n-1}}^\wr \ar[r]^-\sim &\mod\uM \ar[d]^\wr\\
\M^{\bot_{n-2}} / [\Omegab \M] \ar[r]^-\sim &\mod\oM.
}$$
By duality, if we denote by $\mod' \uM$ (resp. $\mod' \oM$) the category of finitely copresented modules over $\uM$ (resp. $\oM$), we get the following commutative diagram:
$$\xymatrix{
{}^{\bot_{n-2}} \M/[\Omega\M]  \ar[d]_{\tau_{n-1}}^\wr  \ar[r]^-\sim &\mod' \uM \ar[d]^\wr \\
\M^{\bot_{n-2}}/[\M]\ar[r]^-\sim &\mod' \oM
}$$
where $\Omega\M$ the class of objects $X\in {\B}$ such that there exists a short exact sequence
$$0\rightarrow X\rightarrow P\rightarrow M\rightarrow 0$$
with $M\in \M$ and $P$ projective.

In Section 2, we collect basic material on exact categories. Section 3 contains main results and their proofs. In Section 4, we discuss several examples.

\section{Preliminaries}

\subsection{Exact categories}

In this section, we briefly review the essential properties of exact categories. For more details, we refer to \cite{2}, \cite[Appendix A]{1} and~\cite{Qui}. Let $\B$ be an additive category. Let us call \emph{weak short exact sequences} of $\B$ pairs of morphisms $(i,d)$ such that $i$ is a kernel of~$d$ and $d$ a cokernel of $i$. Let $\mathcal S$ be a class of weak short exact sequences of $\B$, stable under isomorphisms, direct sums and direct summands. If a weak short exact sequence $(i,d)$ is in $\mathcal S$, we call it a \emph{short exact sequence}, denote it by
$$0 \rightarrow X \xrightarrow{i} Y \xrightarrow{d} Z \rightarrow 0,$$
and we call $i$ an \emph{inflation} and $d$ a \emph{deflation}. The class $\mathcal S$ is called an \emph{exact structure} on $\B$ (and $\B$ is said to be an \emph{exact category}) if it satisfies the following properties, equivalent to the original axioms:

\begin{itemize}
\item Identity morphisms are inflations and deflations.

\item The composition of two inflations (resp. deflations) is an inflation (resp. deflation).

\item If $0 \rightarrow X \xrightarrow{i} Y \xrightarrow{d} Z \rightarrow 0$ is a short exact sequence, for any morphisms $f : Z' \rightarrow Z$ and $g: X \rightarrow X'$, there are commutative diagrams
$$\xymatrix{
0 \ar[r] & X \ar@{=}[d] \ar[r] & Y' \ar[d]_{f'} \ar[r]^{d'} \ar@{}[dr]|{(1)} & Z' \ar[d]^f \ar[r] & 0 \\
0 \ar[r] & X \ar[r] & Y \ar[r]_d & Z \ar[r] & 0 \\
} \quad \xymatrix{
0 \ar[r] & X \ar[d]_{g} \ar@{}[dr]|{(2)} \ar[r]^i & Y \ar[d]^{g'} \ar[r] & Z \ar@{=}[d] \ar[r] & 0 \\
0 \ar[r] & X' \ar[r]_{i'} & Y' \ar[r] & Z \ar[r] & 0 \\
}$$
the rows of which are short exact sequences, the square $(1)$ being a pull-back and $(2)$ being a push-out. In this case, we easily deduce from previous points that
\begin{align*}
 & 0 \rightarrow Y' \xrightarrow{\svecv{f'}{d'}} Y \oplus Z' \xrightarrow{\svech{-d}{f}} Z \rightarrow 0 \\
 \text{and} \quad & 0 \rightarrow X \xrightarrow{\svecv{g}{i}} X' \oplus Y \xrightarrow{\svech{-i'}{g'}} Y' \rightarrow 0
\end{align*}
are short exact sequences.
\end{itemize}

From the definition of the exact category, we get the following lemma.

\begin{lem}\label{2.16}
If $0 \rightarrow X \xrightarrow{i} Y \xrightarrow{d} Z \rightarrow 0$ and $0\rightarrow Y\xrightarrow{f} W\rightarrow V\rightarrow 0$ are two short exact sequences, then there is a commutative diagram of short exact sequences
$$\xymatrix{
&&0\ar[d] &0\ar[d]\\
0 \ar[r] &X \ar@{=}[d] \ar[r]^{i} &Y \ar[d]^{f} \ar[r]^{d} &Z \ar[d] \ar[r] &0\\
0 \ar[r] &X \ar[r] &W \ar[r]^g \ar[d] &U \ar[r] \ar[d] &0\\
&&V \ar@{=}[r] \ar[d] &V \ar[d]\\
&&0 &0
}$$
where the upper-right square is both a push-out and a pull-back.
\end{lem}

\begin{proof}
Since $f:Y\rightarrow W$ is an inflation, there exists a push-out
$$\xymatrix{
Y \ar[d]^{f} \ar[r]^{d} &Z \ar[d]\\
W \ar[r]^g &U.}$$
Then we get a short exact sequence $0\rightarrow Y\rightarrow Z\oplus W\rightarrow U\rightarrow 0$, and thus the above diagram is also a pull-back. Now, if $g':V'\rightarrow W$ is a morphism such that $gg'=0$, then by the definition of a pull-back diagram, there exists a morphism $h:V'\rightarrow Y$
$$\xymatrix{
V' \ar@/_/[ddr]_{g'} \ar@{.>}[dr]^h \ar@/^/[drr]^0\\
&Y \ar[d]^{f} \ar[r]^{d} &Z \ar[d]\\
&W \ar[r]^g &U}$$
such that $g'=fh$ and $dh=0$. Thus $h$ factors through $i$, and therefore $g'$ factors through $fi$.
Finally, as $fi$ is an inflation, $fi$ is a kernel of $g$. If $l: W \rightarrow N$ satisfies $lfi=0$, then there exists a unique morphism $k: Z\rightarrow N$ such that $lf=kd$, since $d$ is the cokernel of $i$. Then by definition of a push-out diagram, there is a unique morphism $s:U\rightarrow N$ such that $l=sg$:
$$\xymatrix{
Y \ar[d]^f \ar[r]^{d} &Z \ar[d] \ar@/^/[ddr]^k\\
W \ar@/_/[drr]_l \ar[r]^g &U \ar@{.>}[dr]^s\\
&&N.}$$
Hence $g$ is the cokernel of $fi$.
\end{proof}

An object $P \in \B$ is said to be \emph{projective} if for any deflation $d$, $\Hom_\B(P, d)$ is surjective. A \emph{projective cover} of an object $X \in \B$ is a deflation $P \rightarrow X$ where $P$ is projective. An object $I \in \B$ is said to be \emph{injective} if for any inflation $i$, $\Hom_\B(i, I)$ is surjective. An \emph{injective envelope} of an object $X \in \B$ is an inflation $X \rightarrow I$ where $I$ is injective. The subcategory of projectives (resp. injectives) is denoted by $\mathcal P$ (resp. $\mathcal I$).

As for abelian categories, $\B$ is said to have \emph{enough injectives} (resp. \emph{enough projectives}) if every object admits an injective envelope (resp. a projective cover).

As in abelian category, if $\B$ has enough injectives (resp. projectives), we can consider injective (resp. projective) resolutions and $\Ext$ functors as right-derived functors of $\Hom$. Thus, for all $X, Y \in \B$, $\Ext_{\B}^1(Y, X)$ parameterizes short exact sequences $0 \rightarrow X \rightarrow Z \rightarrow Y \rightarrow 0$ up to equivalence.

\subsection{Cluster tilting subcategories}

From now on, let $\B$ be a $\Hom$-finite Krull-Schmidt $k$-linear exact category with enough injectives and projectives.

Recall that a subcategory $\mathcal{M}$ is called \emph{contravariantly} (resp. \emph{covariantly}) \emph{finite}
in $\mathcal{B}$ if any object $B$ of $\mathcal{B}$ admits
a \emph{right} (resp. \emph{left}) $\mathcal{M}$-\emph{approximation}, that is a morphism $f: M\rightarrow B$ (resp. $B\rightarrow M$), where $M \in \M$ such that every morphism from $\M$ to $B$ (resp. from $B$ to $\M$) factors through $f$.

A subcategory $\M$ of $\B$ is called $n$-\emph{rigid} if
$$\forall i\in \{1,2,...,{n-1}\}, \, \Ext^i_{\B}(\M,\M)=0.$$

$\M$ is called $n$-\emph{cluster tilting}, if it satisfies the following conditions:

\begin{enumerate}
\item $\mathcal{M }$ is contravariantly finite and covariantly finite in $\B$,
\item $X\in \mathcal{M }$ if and only if $\Ext^i_{\B}(X,\mathcal{M })=0$, $\forall i\in \{1,2,...,{n-1}\}$,
\item $X\in \mathcal{M }$ if and only if $\Ext^i_{\B}(\mathcal{M },X)=0$, $\forall i\in \{1,2,...,{n-1}\}$.
\end{enumerate}

Clearly all $n$-cluster tilting subcategories are $n$-rigid. The unique 1-cluster tilting subcategory is $\B$. By definition, if $\M$ is $n$-cluster tilting, then ${\mathcal P} \subseteq \M$, ${\mathcal I} \subseteq \M$.

\subsection{The categories of modules over (co-)stable subcategories of $\B$}

For any object $X,Y\in \B$ and a full subcategory $\mathcal C$ of $\B$, denote by $[{\mathcal C}](X,Y)$ the set of morphisms in $\Hom_{\B}(X,Y)$ which factor through objects of $\mathcal C$. If $\mathcal{P} \subseteq \mathcal{C}$ (resp. $\mathcal{I} \subseteq \mathcal{C}$), the \emph{(co-)stable category} $\uC$ (resp. $\oC$) of $\C$ is the quotient category $\C/[{\mathcal P}]$ (resp. $\C/[{\mathcal I}]$), i.e. the category which has the same objects than $\C$ and morphisms are defined as
$$\underline \Hom_{\C}(X,Y):=\Hom_{\C}(X,Y)/[{\mathcal P}](X,Y)$$
$$\text{(resp. }\overline \Hom_{\C}(X,Y):=\Hom_{\C}(X,Y)/[{\mathcal I}](X,Y)\text{)}.$$
We denote the residue class of any morphism $f\in \Hom_{\C}(X,Y)$ by $\underline f$ (resp. $\overline f$) in $\uC$ (resp. $\oC$). Denote by $\Mod {\mathcal C}$ the category of contravariant additive functors from ${\mathcal C}$ to $\mod k$ for any category $\mathcal C$. Let $\mod {\mathcal C}$ be the full subcategory of $\Mod{\mathcal C}$ consisting of objects $A$ admitting an exact sequence:
$${\Hom_{\mathcal C}(-,C_1)}\s{\beta}{\rightarrow}{\Hom_{\mathcal C}(-,C_0)}\s{\alpha}\rightarrow A\rightarrow0$$
where $C_0,C_1\in \mathcal C$.

\begin{rem}
 $\Hom_{\mathcal C}(-,C)$ is projective in $\Mod \mathcal C$ for any $C\in \mathcal C$.
\end{rem}

\begin{lem}\label{2.2}
For any contravariantly finite subcategory $\mathcal C$ of $\B$ which contains $\mathcal P$, $\mod\underline {\mathcal C}$ is an abelian category.
\end{lem}

\begin{proof}
It is enough to show that $\underline {\mathcal C}$ has pseudokernels (see \cite[{\S2}]{3}). Consider a morphism $f \in \Hom_{\B}(C_1,C_0)$. Let $$0\rightarrow K\rightarrow P_0\s{f'}\rightarrow C_0\rightarrow 0$$ be a short exact sequence with $P_0$ projective. Then there is a commutative diagram of exact sequences:
$$\xymatrix{
0 \ar[r] &K \ar@{=}[d] \ar[r] &L \ar[d]^{g} \ar[r]^{g'} &C_1 \ar[d]^{f} \ar[r] &0\\
0 \ar[r] &K \ar[r] &P_0 \ar[r]^{f'} &C_0 \ar[r] &0}$$
where the right square is a pull-back and a short exact sequence:

$$0 \rightarrow L \xrightarrow{\svecv{-g'}{g}} C_1 \oplus P_0 \xrightarrow{\svech{f}{f'}} C_0 \rightarrow 0.$$

Let $k:C_L\rightarrow L$ be a right ${\mathcal C}$-approximation of $L$. We claim that $\underline {g'k}$ is a pseudokernel of $\underline f$. If $\underline {fh}=0$ for a morphism $\underline h \in \underline \Hom_{\B}(C,C_1)$ with $C \in \mathcal C$, then there exists an commutative diagram:
$$\xymatrix{
C \ar[d]^h \ar[r]^{h_1} &P \ar[d]^{h_2}\\
C_1 \ar[r]^f &C_0
}$$
where $P$ is projective. As $f'$ is a deflation, there exists a morphism $h':P\rightarrow P_0$ such that $f'h'=h_2$. Since $$\vecv{h}{-h'h_1} \in \Hom_{\B}(C,C_1\oplus P_0)$$ and $$\vech{f}{f'}\circ\vecv{h}{-h'h_1}=fh-f'h'h_1=fh-h_2h_1=0,$$ $\vecv{h}{-h'h_1}$ factors through $\vecv{-g'}{g}$. Since $k$ is a right ${\mathcal C}$-approximation of $L$, we have the following commutative diagram:
$$\xymatrix{
&C \ar[ld] \ar[d] \ar[rd]^-{\svecv{h}{-h'h_1}}\\
C_L \ar[r]_k &L \ar[r]_-{\svecv{-g'}{g}} &{C_1\oplus P_0}}$$
and $\underline h$ factors though $\underline {g'k}$.
\end{proof}

\section{Main results}

\subsection{Quotient category of $\M_L$ by a rigid subcategory $\M$}

In this subsection, we assume that $\M$ is a rigid subcategory of $\B$ which contains $\mathcal P$. Recall that $\M_L$ is the subcategory of objects $X$ of $\B$ which admit a short exact sequence $0\rightarrow X\xrightarrow{d^0}M^0\xrightarrow{d^1} M^1\rightarrow 0$
where $M^0,M^1\in \M$. Now we consider the functor
\begin{align*}
H: \text{ }&{\M_L}\rightarrow \Mod\uM\\
   &X\mapsto \Ext^1_{\B}(-,X)|_{\M}
\end{align*}

From now on, we will consider the quotient category $\M_L/[\M]$ and we will denote by $[f]$ the residue class in $\M_L/[\M]$ of any morphism $f$ of $\M_L$.

Let $\pi:{\M_L} \rightarrow{\M_L / [\M]}$ be the projection functor. By definition of a rigid subcategory, $HX=0$ if $X \in \M$. Hence, by the universal property of $\pi$, there exists a functor $F: \M_L/[\M] \rightarrow \Mod \uM$ such that the following diagram commutes:
$$\xymatrix{
{\M_L} \ar[d]^{\pi} \ar[rd]^H\\
{\M_L/[\M]} \ar@{.>}[r]^-F &{\Mod\uM}.}$$

\begin{lem}\label{2.4}
For any short exact sequence
$$0\rightarrow X\xrightarrow{d^0}M^0\xrightarrow{d^1} M^1\rightarrow 0$$
where $M^0,M^1\in \M$, there is an exact sequence in $\Mod\uM$
$$\underline \Hom_{\M}(-,M^0)\rightarrow \underline \Hom_{\M}(-,M^1)\rightarrow HX\rightarrow 0.$$
Thus, $FX=HX \in \mod\uM$.
\end{lem}

\begin{proof}
 Applying $\Hom_{\B}(M,-)$ where $M\in \M$ to the short exact sequence, we get a long exact sequence:
 $$\Hom_{\B}(M,M^0)\xrightarrow{\Hom_{\B}(M,d^1)} \Hom_{\B}(M,M^1)\xrightarrow{\delta} \Ext^1_{\B}(M,X)\rightarrow 0.$$
 This means that the following sequence in $\mod{\M}$ is exact:
 $$\Hom_{\M}(-,M^0)\rightarrow \Hom_{\M}(-,M^1)\rightarrow HX \rightarrow 0.$$

 As $d^1$ is a deflation, any morphism $g \in [{\mathcal P}](M,M^1)$ factors through $d^1$. Hence $\delta(g)=0$. Thus $\delta$ induces a morphism $\underline \delta:\underline \Hom_{\B}(M,M^1)\rightarrow HX(M)=\Ext^1_{\B}(M,X)$: $\underline \delta(\underline f)=\delta(f)$. For any $M\in \M$, $\underline \delta$ is surjective. Moreover, since $\delta \circ \Hom_{\B}(M,d^1)=0$, we get $\underline \delta \circ \underline {\Hom}_{\B}(M,d^1)=0$. If $\underline \delta(\underline g)=0$ for a morphism $g \in \Hom_{\B}(M,M^1)$, then by definition $\delta(g)=0$, there exists a morphism $h \in \Hom_{\B}(M,M^0)$ such that $g=fh$. Hence $\underline g=\underline {fh}$. Thus we get an exact sequence in $\mod\uM$
 $$\underline \Hom_{\M}(-,M^0)\rightarrow \underline \Hom_{\M}(-,M^1)\rightarrow HX\rightarrow 0.$$

 The last argument can be understood as the right exactness of the functor $\uM \otimes_\M - : \mod \M \rightarrow \mod \uM$.
\end{proof}

\begin{thm}\label{2.5}
The functor $F:$ $\M_L/[\M] \rightarrow \mod\uM$ is an equivalence of categories.
\end{thm}

\begin{proof}
$\bullet$~Let us prove that $F$ is full:\\
Let $X,Y\in \M_L$ and $\alpha \in \Hom_{\mod\uM}(FX,FY)$. Consider the following short exact sequences:
$$0\rightarrow X\xrightarrow{d^0} M^0\xrightarrow{d^1} M^1\rightarrow 0$$
$$0\rightarrow Y\xrightarrow{e^0} {N^0}\xrightarrow{e^1} {N^1}\rightarrow 0$$
where $M^0,M^1,N^0,N^1\in \M$. By Lemma \ref{2.4}, and because $\Hom_{\M}(-,M^1)$ is projective in $\Mod\uM$, we get the following diagram with exact rows in $\mod{\uM}$:
$$\xymatrix{
\underline\Hom_{\M}(-,M^0) \ar[d]^{\nu_0} \ar[rr]^-{\underline\Hom_{\M}(-,\underline {d^1})} &&{\underline\Hom_{\M}(-,M^1)} \ar[d]^{\nu_1} \ar[r] &HX \ar[d]^{\alpha} \ar[r] &0\\
\underline\Hom_{\M}(-,{N^0}) \ar[rr]^-{\underline\Hom_{\M}(-,\underline {e^1})} &&{\underline\Hom_{\M}(-,N^1)} \ar[r] &HY \ar[r] &0.
}$$
By Yoneda's Lemma, for $i=0,1$, there is a morphism $\underline{f^i} \in \uHom_\M (M^i, N^i)$ such that $\nu_i=\underline\Hom_{\M}(-,\underline {f^i})$ and $\underline {e^1 f^0}=\underline {f^1d^1}$. Then there exists a projective $P \in \M$, $a \in \Hom_\M(M^0, P)$ and $b \in \Hom_\M(P, N^1)$ such that $f^1 d^1 - e^1 f^0 = ba$. Thus, as $e^1$ is a deflation, there exists $c \in \Hom_\M(P, N^0)$ such that the following diagram commutes:
$$\xymatrix{
&&M^0 \ar[dl]_a \ar[dd]^{f^1d^1-e^1f^0}\\
&P \ar[dr]^b \ar[dl]_c\\
N^0 \ar[rr]^-{e^1} &&N^1.}$$
Let $g^0=f^0+ca$, then $\underline {g^0} = \underline {f^0}$ and $e^1 g^0=f^1 d^1$. Therefore, there exists $f: X\rightarrow Y$ such that $e^0 f= g^0d^0$, and $\alpha=F([f])$.

$\bullet$~Let us prove that $F$ is faithful:\\
Let $f:X\rightarrow Y$ be a morphism in $\M_L$. Since $d^0$ is a left $\M$-approximation, there exists a commutative diagram of short exact sequences
$$\xymatrix{
0 \ar[r] &X \ar[d]^f \ar[r]^{d^0} &M^0 \ar[d]^{f^0} \ar[r]^{d^1} &M^1 \ar[d]^{f^1} \ar[r]&0\\
0 \ar[r] &Y \ar[r]^{e^0} &{N^0} \ar[r]^{e^1} &{N^1} \ar[r] &0,}$$
and a commutative diagram in $\mod\uM$
$$\xymatrix{
\underline\Hom_{\M}(-,M^0) \ar[d]_{\underline\Hom_{\M}(-,\underline{f^0})} \ar[rr]^-{\underline\Hom_{\M}(-,\underline{d^1})} &&{\underline\Hom_{\M}(-,M^1)} \ar[d]^{\underline\Hom_{\M}(-,\underline{f^1})} \ar[r] &HX \ar[d]^{H(f)} \ar[r] &0\\
\underline\Hom_{\M}(-,{N^0}) \ar[rr]_-{\underline\Hom_{\M}(-,\underline{e^1})} &&{\underline\Hom_{\M}(-,\underline N^1)} \ar[r] &HY \ar[r] &0.}$$
Assume that $F(\underline f)=0$. Then $H(f)=0$ and thus, there exists a morphism $\beta:\underline\Hom_{\M}(-, M^1)\rightarrow \underline\Hom_{\M}(-,N^0)$ such that $\underline\Hom_{\M}(-,\underline{f^1})=\underline\Hom_{\M}(-,\underline{e^1}) \circ \beta$. By Yoneda's Lemma, there exists a morphism $g: M^1\rightarrow N^0$ such that $\beta=\underline\Hom_{\M}(-,\underline g)$ and $\underline {f^1}=\underline{e^1 g}$. Then, as before, we can complete the following commutative diagram
$$\xymatrix{
&&M^1 \ar[dl]_{a'} \ar[dd]^{f^1-e^1 g}\\
&P' \ar[dr]^{b'} \ar[dl]_{c'}\\
N^0 \ar[rr]^-{e^1} &&N^1}$$
where $P'$ is projective. Hence $e^1(c'a'+g)d^1 = f^1 d^1 = e^1 f^0$ and there exists a morphism $g':M^0\rightarrow Y$ such that $e^0g'=f^0-(c'a'+g)d^1$ and $e^0(f-g'd^0)=0$. Since $e^0$ is monic, $f=g'd^0$, thus $[f]=0$ in $\M_L/[\M]$.

$\bullet$~Let us prove that $F$ is dense:\\
For any object $C \in\mod\uM$, there is an exact sequence in $\mod\uM$:
$${\underline \Hom_{\M}(-,M_1)}\xrightarrow{\beta} {\underline \Hom_{\M}(-,M_0)} \rightarrow C\rightarrow 0$$
where $M_0,M_1 \in \M$. Since $$\Hom_{\mod\uM}(\underline \Hom_{\M}(-,M_1),\underline \Hom_{\M}(-,M_0)) \simeq \underline \Hom_{\M}(M_1,M_0)$$
by Yoneda's Lemma, there exists $f:M_1\rightarrow M_0$ such that $\beta =\underline \Hom_{\M}(-,\underline f)$.
Let $$0\rightarrow K\rightarrow P_0\xrightarrow{f'} M_0\rightarrow 0$$ be a short exact sequence with $P_0$ projective. From the proof of Lemma \ref{2.2}, there exists a short exact sequence
$$0\rightarrow L\rightarrow M_1\oplus P_0\xrightarrow{\svech{f}{f'}} M_0\rightarrow 0.$$

Therefore $L\in \M_L$ and there is an exact sequence:
$$\underline \Hom_{\M}(-,M_1)\xrightarrow{\beta}\underline \Hom_{\M}(-,M_0)\rightarrow FL\rightarrow0.$$

Since both $FL$ and $C$ are cokernels of $\beta$, $C \simeq FL$.
\end{proof}

By this theorem and Lemma \ref{2.2}, we get the following corollary:
\begin{cor} \label{ab}
If $\M$ is rigid and contravariantly finite, then $\M_L/[\M]$ is abelian.
\end{cor}

\subsection{Quotient category of $\M_R$ by $\Omegab \M$}

In this subsection we assume that $\M$ is a rigid subcategory of $\B$ which contains $\mathcal I$. Recall that $\M_R$ is the subcategory of objects $Y$ which admit a short exact sequence $$0\rightarrow N_1\xrightarrow{d_1'}N_0\xrightarrow{d_0'} Y\rightarrow 0$$ where $N_0$, $N_1\in \M$ and $\Omegab \M$ the full subcategory of objects $X\in {\B}$ such that there exists a short exact sequence $0\rightarrow M\rightarrow I\rightarrow X\rightarrow 0$
where $M\in {\M}, I \in {\mathcal I} $.

We denote
\begin{align*}
  K:{\M_R} & \rightarrow \Mod\oM \\
  X &\mapsto \overline \Hom_{\B}(-,X)|_{\oM}.
\end{align*}

 Let $\pi':\M_R\rightarrow {\M_R / [\Omegab \M]}$ be the projection functor. By the universal property of $\pi'$, there is a functor $G:{\M_R / [\Omegab \M]}\rightarrow \Mod{\oM}$ which makes the following diagram commute:
$$\xymatrix{
{\M_R} \ar[d]^{\pi'} \ar[rd]^{K}\\
\M_R/[\Omegab \M] \ar@{.>}[r]^-{G} &{\Mod\oM}.}$$

\begin{lem}\label{2.6}
For every $X\in\M_R$, $G(X)=K(X)\in \mod\oM$. More precisely, for every short exact sequence
$$0\rightarrow M_1\s{d_1}\rightarrow  M_0\s{d_0}\rightarrow  X\rightarrow 0$$
where $M_1,M_0 \in \M$, there is an exact sequence
$$K(M_1)\xrightarrow{K({d_1})} K(M_0)\xrightarrow{K({d_0})} K(X)\rightarrow 0.$$
\end{lem}

\begin{proof}
For any nonzero object $M \in \M$, we have an exact sequence
$$0\rightarrow \Hom_{\B}(M,M_1)\rightarrow \Hom_{\B}(M,M_0)\xrightarrow{\Hom_{\B}(M,d_0)} \Hom_{\B}(M,X)\rightarrow 0$$
by applying $\Hom_{\B}(M,-)$. Hence $d_0$ is a right $\M$-approximation and $K({d_0}): K(M_0) \rightarrow K(X)$ is surjective. If a morphism $h \in \Hom_{\B}(M,M_0)$ satisfies $\overline {d_0h}=0$ in $\oB$, then there exists an object $I \in \mathcal I$ and a commutative diagram
$$\xymatrix{
&&M \ar[r]^a \ar[d]^h &I \ar[d]^b\\
0 \ar[r] &{M_1} \ar[r]^{d_1} &M_0 \ar[r]^{d_0} &X \ar[r] &0.
}$$

Since $I \in \M$, there is a morphism $c:I\rightarrow M_0$ such that $b=d_0c$. Hence $d_0ca=ba=d_0h$, and thus $d_0(h-ca)=0$. Therefore, $h-ca$ factors through $d_1$. Finally $\overline h$ factors through $\overline {d_1}$.
\end{proof}

\begin{thm}\label{2.7}
The functor $G:{\M_R / [\Omegab \M]}\rightarrow \mod{\oM}$ is an equivalence of categories.
\end{thm}

\begin{rem}
 In contrast to Corollary \ref{ab}, we cannot conclude that $\M_R / [\Omegab \M]$ is abelian. Indeed, it is unlikely that a result analogous to Lemma \ref{2.2} is true for $\mod \oM$.
\end{rem}

\begin{proof}
Let $X$, $Y \in M_R$. By definition, there exist short exact sequences
$$0\rightarrow M_1\xrightarrow{d_1} M_0\xrightarrow{d_0} X\rightarrow 0$$
$$0\rightarrow N_1\xrightarrow{d_1'} N_0\xrightarrow{d_0'} Y\rightarrow 0.$$

$\bullet$~Let us show that $K$ (and therefore $G$) is full:\\
Let $\alpha:K(X)\rightarrow K(Y)$ be a morphism in $\mod\oM$.
By Lemma \ref{2.6}, and because $K(M_0)$ is projective, we can form the following commutative diagram
$$\xymatrix{
K(M_1) \ar[d]^{\gamma} \ar[r]^{K(d_1)} &K(M_0) \ar[d]^{\beta} \ar[r]^{K(d_0)} &K(X) \ar[d]^{\alpha} \ar[r] &0\\
K(N_1) \ar[r]_{K(d'_1)} &K(N_0) \ar[r]_{K(d'_0)} &K(Y) \ar[r] &0.
}$$

By Yoneda's Lemma, there exist $g:M_0\rightarrow N_0$ and $h:M_1\rightarrow N_1$ such that $\beta=K(g)$ and $\gamma=K(h)$. Hence $\overline {gd_1}=\overline {d_1'h}$, and therefore $gd_1-d_1'h$ factors through some object $I \in \mathcal I$. As $d_1$ is an inflation, we can complete the following commutative diagram:
$$\xymatrix{
M_1 \ar[dr]_a \ar[dd]_-{gd_1-d_1'h} \ar[rr]^{d_1} &&{M_0} \ar[dl]^c\\
&I \ar[dl]^b\\
N_0.}$$
Thus, putting $g'=g-bc$, we get $g'd_1=d_1'h$. Hence there is a morphism $f:X\rightarrow Y$ such that the diagram
$$\xymatrix{
0 \ar[r] &M_1 \ar[d]^h \ar[r]^{d_1} &M_0 \ar[d]^{g'} \ar[r]^{d_0} &X \ar[d]^f \ar[r] &0\\
0 \ar[r] &N_1 \ar[r]^{d_1'} &N_0 \ar[r]^{d_0'} &Y \ar[r] &0}$$
commutes. Finally, $\alpha=K(f)$.

$\bullet$~Let us show that $G$ is faithful:\\
Let $M_0\s{i}\rightarrow I_M$ be an injective envelope of $M_0$. Let $f:X \rightarrow Y$ be a morphism of $\M_R$. If $G([f])=0$ then $K(f)=0$. Hence, $fd_0$ factors through some object in $\mathcal I$. Therefore, we have a commutative diagram
$$\xymatrix{
M_0 \ar[d]_i \ar[r]^{d_0} &X \ar[d]^f\\
I_M \ar[r]_{e} &Y.
}$$
Then, if $N$ is a push-out of $i$ and $d_0$, we can form the commutative diagram
$$\xymatrix{
M_0 \ar[d]_i \ar[r]^{d_0} &X \ar[d]^k \ar@/^/[ddr]^f\\
I_M \ar@/_/[drr]_e \ar[r]^l &N \ar[dr]^r\\
&&Y.}$$
By Lemma \ref{2.16}, there is a short exact sequence
$$0\rightarrow M_1\rightarrow I_M \rightarrow N\rightarrow 0.$$

Hence $N \in \Omegab \M$ and then $[f]=0$ in $\M_R/[\Omegab \M]$.

$\bullet$~Let us show that $G$ is dense:\\
For any object $C \in \mod\oM$, there is an exact sequence
$$K(M_1)\xrightarrow{K(d)}K(M_0)\rightarrow C\rightarrow 0.$$
First, let us form a push-out diagram
$$\xymatrix{
M_1 \ar[d]^s \ar[r]^{d} &M_0 \ar[d]\\
I' \ar[r] &L}$$
where $s$ is an injective envelope of $M_1$. Then
$$0\rightarrow M_1\rightarrow I'\oplus M_0 \rightarrow L\rightarrow 0$$
is a short exact sequence.
By Lemma \ref{2.6}, we get an exact sequence
$$K(M_1)\rightarrow K(M_0)\rightarrow K(L)\rightarrow 0.$$
Hence $C\simeq K(L)=G(L)$.
\end{proof}

If we denote $\oM^\bot=\{X \in {\M_R}\text{ }| \text{ }\overline \Hom_{\B}({\M},X)=0 \}$, we get the following corollary:
\begin{cor}\label{2.17}
$\Omegab \M=\oM^\bot$.
\end{cor}

\begin{proof}
By definition, $\Omegab \M\subseteq \oM^\bot$. If $X\in\oM^\bot$, then $G(X)=K(X)=0$. Since $G$ is faithful, $X\in \Omegab \M$.
\end{proof}

\subsection{Case of $n$-cluster tilting subcategories and AR translation}

In this subsection, we assume that $n>1$ and $\M$ is an $n$-cluster tilting subcategory.

Define
\begin{align*}
 & {}^{\bot_{n-2}} \M = \{ X\in {\B} \,|\, \forall i \in \{1, \dots, n-2\}, \Ext_{\B}^i(X, {\M})=0 \} \\
\text{and} \quad & \M^{\bot_{n-2}} = \{ X\in {\B} \,|\, \forall i \in \{1, \dots, n-2\}, \Ext_{\B}^i({\M}, X)=0 \}
\end{align*}

\begin{rem}\label{Mpex}
 The categories ${}^{\bot_{n-2}} \M$ and $\M^{\bot_{n-2}}$ are extension closed and thus exact subcategories of $\B$.
\end{rem}

\begin{prop}\label{2.3}
The following equalities hold: $${^{\bot_{n-2}}{\M}}=\M_L \quad \text{and} \quad {{\M}^{\bot_{n-2}}}=\M_R.$$
\end{prop}

\begin{proof}
We only prove the first equality. Let $X \in {}^{\bot_{n-2}} \M$ and $0\rightarrow X\xrightarrow{d_0} I_0\rightarrow C\rightarrow 0$ be a short exact sequence with $I_0 \in \mathcal I$. By definition of an exact category, we get the following commutative diagram
$$\xymatrix{
0 \ar[r] &X \ar[d]^{d_X} \ar[r]^{d_0} &I_0 \ar[d] \ar[r] &C \ar@{=}[d] \ar[r] &0\\
0 \ar[r] &M_X \ar[r] &C' \ar[r] &C \ar[r] &0}$$
where $d_X$ is a left $\M$-approximation and the first square is a push-out. Moreover, we have a short exact sequence
$$0\rightarrow X\xrightarrow{\svecv{d_X}{d_0}} M_X\oplus I_0\rightarrow C'\rightarrow 0.$$
If $M \in \M$, we have an exact sequence
$$\Hom_{\B}(M_X\oplus I_0,M)\rightarrow \Hom_{\B}(X,M)\rightarrow \Ext^1_{\B}(C',M)\rightarrow 0$$ since $\Ext^1_{\B}(M_X\oplus I_0,M)=0$. Moreover, $\Hom_{\B}(d_X,M)$ is surjective, since $d_X$ is a left $\M$-approximation. Hence $\Ext^1_{\B}(C',M)=0$.

Moreover, applying $\Hom_{\B}(-,M)$  where $M\in \M$ to the former short exact sequence, we get a long exact sequence
\begin{align*}
\cdots &\rightarrow \Ext_{\B}^i(M_X\oplus I_0,M) \rightarrow \Ext_{\B}^i(X,M)\rightarrow \Ext_{\B}^{i+1}(C',M)\\
                                               &\rightarrow \Ext_{\ B}^{i+1}(M_X\oplus I_0,M)\rightarrow \cdots
\end{align*}
for $\forall i\in \{1,2,...,{n-2}\}$. Since $\Ext_{\B}^i(M_X \oplus I_0,M)=\Ext_{\B}^{i+1}(M_X \oplus I_0,M)=0$, we get $\Ext_{\B}^i(X,M)\simeq \Ext_{ B}^{i+1}(C',M)$.

 Thus, as $X \in {^{\bot_{n-2}}{\M}}$, $\Ext^i_{\B}(C',M)=0$ if $1<i<n-1$. As $\M$ is $n$-cluster tilting, $C' \in \M$. Finally, ${}^{\bot_{n-2}} \M \subseteq \M_L$.

Suppose that $X \in \M_L$. Then it admits a short exact sequence
$$0\rightarrow X\rightarrow M^0\rightarrow M^1\rightarrow 0$$
where $M^0,M^1 \in \M$. Applying $\Hom_{\B}(-,M)$ where $M\in \M$, we get a long exact sequence
$$0 = \Ext_{\B}^i(M^0,M) \rightarrow \Ext_{\B}^i(X,M)\rightarrow \Ext_{\B}^{i+1}(M^1,M) = 0$$
for $\forall i\in \{1,2,...,{n-2}\}$. Therefore, $X \in {^{\bot_{n-2}}{\M}}$.
\end{proof}

For any object $M \in \M$, we fix a short exact sequence $0\rightarrow M\rightarrow I_M\rightarrow N\rightarrow 0$ with $I_M \in \mathcal I$ and denote $N$ by $\Omegab M$. Notice that $\Omegab M$ does not depend on $I_M$ as an object of $\overline{\Omegab \M}$. Moreover, if $f:M\rightarrow M'$ is a morphism in $\M$ then we can form a commutative diagram:
$$\xymatrix{
0 \ar[r] &M \ar[d]^f \ar[r] &I_M \ar[d]^g \ar[r] &{\Omegab \M} \ar[d]^h \ar[r] &0\\
0 \ar[r] &M' \ar[r] &I_{M'} \ar[r] &{\Omegab {\M}'} \ar[r] &0}$$
where $\Omegab \overline f:=\overline h$ is uniquely determined by the morphism $\overline f$ in $\oM$. Then we get

\begin{prop}\label{2.10}
With the previous notation, $\Omegab :\oM\rightarrow \overline{\Omegab \M}$ is an equivalence of categories.
\end{prop}

Now, we assume that $\B$ has an AR translation $\tau: \uB \rightarrow \oB$ with reciprocal $\tau^-$. Following \cite{IyamaHigher}, we define $(n-1)$-AR translations  $$\tau_{n-1}: \underline{{}^{\bot_{n-2}} \mathcal P} \rightarrow \overline{\mathcal I^{\bot_{n-2}}} \quad \text{and} \quad \tau_{n-1}^-: \overline{\mathcal I^{\bot_{n-2}}} \rightarrow \underline{{}^{\bot_{n-2}} \mathcal P}$$ by $\tau_{n-1} = \tau \Omega^{n-2}$ and $\tau_{n-1}^- = \tau \Omegab^{n-2}$ (where $\Omega$ is the syzygy functor). In fact, the only property we need for these functors is that, if $X \in {}^{\bot_{n-2}} \mathcal P$ and $Y \in \mathcal I^{\bot_{n-2}}$, the following functorial isomorphisms hold:
\begin{enumerate}
 \item $\Ext^{n-1}_{\B}(X,Y)\simeq \D\overline\Hom_{\B}(Y,\tau_{n-1} X)\simeq \D\underline\Hom_{\B}(\tau^{-}_{n-1} Y,X)$,
 \item $\forall i\in \{1,2,...,{n-2}\}$, \\ $\Ext^{n-1-i}_{\B}(X,Y)\simeq \D\Ext^i_{\B}(Y, \tau_{n-1}X)\simeq \D\Ext^i_{\B}(\tau^{-}_{n-1}Y, X)$
\end{enumerate}
where $D = \Hom_k(-,k)$. This is a weak version of \cite[Theorem 1.5]{IyamaHigher}.

From this, we deduce easily that $\tau_{n-1}$ induces an equivalence from $\underline {{^{\bot_{n-2}}}\M}$ to $\overline {\M^{\bot_{n-2}}}$ the inverse of which is $\tau_{n-1}^{-1} = \tau_{n-1}^-$.

Remark that
\begin{align*}
X \in \M &\Leftrightarrow \Ext_{\B}^i(X,\M)=0, \text{ } \forall i\in \{1,2,...,{n-1}\}\\
            &\Leftrightarrow
	      \left\{\begin{array}{l}
	              \overline\Hom_{\B}(\M,\tau_{n-1} X)=0 \\
		      \Ext^i_{\B}(\M, \tau_{n-1}X)= 0 \quad \text{for all } i\in \{1,2,...,{n-2}\}
	             \end{array}\right.\\
            &\Leftrightarrow \tau_{n-1}X \in \M^{\bot_{n-2}} \cap \oM^\bot.
\end{align*}

Moreover, as $\oM^\bot = \Omegab \M \subseteq \M^{\bot_{n-2}}$, $X \in \M \Leftrightarrow \tau_{n-1} X \in \Omegab  \M$.

Now $X \in \mathcal P$ implies that $\Ext^{n-1}_{\B}(X,\B)=0$, then $\overline\Hom_{\B}(\B,\tau_{n-1} X)=0$, which means $\tau_{n-1} X \in \mathcal I$. Dually $X\in \mathcal I$ implies that $\tau^{-1}_{n-1} X \in \mathcal P$. Hence $X \in \mathcal P \Leftrightarrow \tau_{n-1} X \in \mathcal I$.
We get the following proposition

\begin{prop}\label{2.11}
The functor $\tau_{n-1}$ induces an equivalence from $\uM$ to $\overline{\Omegab \M}$ and an equivalence from $^{\bot_{n-2}}\M/[\M]$ to $\M^{\bot_{n-2}}/[\Omegab \M]$.
\end{prop}

Denote by $\Omegab ^{-1}$ the inverse of $\Omegab :\oM\rightarrow \overline{\Omegab \M}$. Then we have

\begin{cor}\label{2.12}
The compositions $\tau^{-1}_{n-1}\circ\Omegab $ and $\Omegab ^{-1}\circ\tau_{n-1}$ induce mutually inverse equivalences between $\oM$ and $\uM$.
\end{cor}

According to this corollary, we can define reciprocal equivalences:
\begin{enumerate}
\item $\mu:\Mod\uM\rightarrow \Mod\oM$, $\mu(C)=C\circ\tau^{-1}_{n-1}\circ\Omegab $,

\item $\mu^{-1}:\Mod\oM\rightarrow \Mod\uM$, $\mu^{-1}(C')=C'\circ\Omegab ^{-1}\circ\tau_{n-1}$.
\end{enumerate}

Thus we get:
\begin{prop}\label{2.13}
The functors $\mu$ and $\mu^{-1}$ induce mutually inverse equivalences between $\mod\uM$ and $\mod\oM$.
\end{prop}

\begin{proof}
We just check that $\mu$ sends an object in $\mod\uM$ to an object in $\mod\oM$. For any object $C\in \mod\uM$, there exists an exact sequence:
$${\underline \Hom_{\M}(-,M_1)}\xrightarrow{\underline{\Hom}_{\M}(-,f)} {\underline \Hom_{\M}(-,M_0)}\rightarrow C\rightarrow0$$
where $M_1,M_0\in {\M}$. Hence we get the following commutative diagram:
$$\xymatrix{
{\underline \Hom_{\B}(\tau^{-1}_{n-1}\Omegab M,M_1)} \ar[r]^-{\simeq} \ar[d]_{\underline \Hom_{\B}(\tau^{-1}_{n-1}\Omegab M, f)} &{\overline \Hom_{\B}(M,\Omegab ^{-1} \tau_{n-1} M_1)} \ar[d]^{\overline \Hom_{\B}(M, \Omegab ^{-1}\tau_{n-1}f)}\\
{\underline \Hom_{\B}(\tau^{-1}_{n-1}\Omegab M,M_0)} \ar[r]^-{\simeq} \ar[d] &{\overline \Hom_{\B}(M,\Omegab ^{-1} \tau_{n-1} M_0)} \ar[d]\\
{C(\tau^{-1}_{n-1}\Omegab M)} \ar@{=}[r] \ar[d] &{C\tau^{-1}_{n-1}\Omegab (M)} \ar[d]\\
0 &0
}$$
where each column is exact and $M \in \M$. Thus $C\tau^{-1}_{n-1}\Omegab  \in \mod\oM$.
\end{proof}

\begin{thm}\label{2.14}
If $\B$ has an $(n-1)$--AR translation $\tau_{n-1}$, then we have a diagram which is commutative up to the equivalence
$$\xymatrix{
^{\bot_{n-2}}\M/[\M] \ar[d]^{\tau_{n-1}} \ar[r]^-{F} &\mod\uM \ar[d]^{\mu}\\
\M^{\bot_{n-2}}/[\Omegab \M] \ar[r]^-{G} &\mod\oM.
}$$
\end{thm}

\begin{proof}
Let $M \in \M$ and $X \in {}^{\bot_{n-2}} \M$. Then, applying $\Hom_\B(X,-)$ to $0 \rightarrow M \rightarrow I_M \xrightarrow{\alpha} \Omegab M \rightarrow 0$, we get an exact sequence:
$$\Hom_{\B}(X,I_M) \xrightarrow{\Hom_\B(X,\alpha)} \Hom_{\B}(X,\Omegab M)\rightarrow \Ext^1_{\B}(X,M)\rightarrow 0.$$
Moreover, if $I \in \mathcal{I}$, applying $\Hom_{\B}(I,-)$ as $I \in \M$ and $\M$ is cluster tilting, we get an exact sequence
$$\Hom_{\B}(I,I_M)\xrightarrow{\Hom_\B(I, \alpha)} \Hom_{\B}(I,\Omegab M)\rightarrow \Ext^1_{\B}(I,M)=0$$
and therefore $$\Ext^1_{\B}(X,M) \simeq \Hom_{\B}(X,\Omegab M)/\Im(\Hom_\B(X,\alpha))=\overline\Hom_{\B}(X,\Omegab M).$$
Hence, if $n=2$, we have
\begin{align*}
 \overline \Hom_{\B}(M,\tau_1 X) &\simeq {\D}\Ext^1_{\B}(X,M)\simeq {\D}\overline \Hom_{\B}(X,\Omegab M)\\
                               &\simeq {\D}\overline \Hom_{\B}(X,\tau_1\tau^{-1}_1\Omegab M)\\
                               &\simeq \Ext^1_{\B}(\tau^{-1}_1\Omegab M,X)
\end{align*}
which implies that $G(\tau_1 X)(M)\simeq \mu (FX)(M)$. As all the isomorphisms used are functorial in $M$ and $X$, $G\tau_1 \simeq \mu F$ as functors.

When $n\geq3$, we get an exact sequence
$$0 = \Ext^i_{\B}(X,I_M)\rightarrow \Ext^{i}_{\B}(X,\Omegab M)\rightarrow \Ext^{i+1}_{\B}(X,M)\rightarrow \Ext^{i+1}_{\B}(X,I_M) = 0$$
for $i\geq1$. Thus we have $\Ext^{i}_{\B}(X,\Omegab M)\simeq \Ext^{i+1}_{\B}(X,M)$. Since $X \in {}^{\bot_{n-2}} \M \subseteq {}^{\bot_{n-2}} \mathcal P$,
\begin{align*}
 \overline \Hom_{\B}(M,\tau_{n-1} X) &\simeq {\D}\Ext^{n-1}_{\B}(X,M)\simeq {\D}\Ext^{n-2}_{\B}(X,\Omegab M)\\
                               &\simeq \Ext^1_{\B}(\Omegab M,\tau_{n-1}X)\\
                               &\simeq \Ext^1_{\B}(\tau^{-1}_{n-1}\Omegab M,X)
\end{align*}
which also implies $G(\tau_{n-1} X)(M)\simeq \mu (FX)(M)$. As all the isomorphisms used are functorial, $G\tau_{n-1} \simeq \mu F$ as functors.
\end{proof}

\section{Examples}

In this section, we explain two sources of examples, one coming from Geometry (Cohen-Macauley modules over a singularity) and the other one coming more directly from representation theory (Auslander algebras). These examples have been studied extensively by Iyama (see \cite{IyamaHigher} and \cite{5}) from the point of view of (higher) cluster tilting theory.

\begin{exm}\label{exCM}
 Let $S = \CC\llbracket X_1, X_2, \dots, X_d \rrbracket$ be a formal power series ring in~$d$ variables and $G$ be a finite subgroup of $\GL_d(\CC)$ without pseudo-reflections acting on $S$. Let $R = S^G$. According to \cite[Proposition 13]{HE}, $R$ is a complete local Cohen-Macaulay ring of Krull-dimension $d$. Then, by \cite[Theorem 2.5]{IyamaHigher}, if $R$ is an isolated singularity, $S$ is a $(d-1)$-cluster tilting object of $\CM(R)$.

 It is known by the end of proof of \cite[Proposition 2.1]{Aus1} together with the discussion after \cite[Proposition 1.1]{Aus2} that $\End_R(S) \simeq S * G$ where $S*G$ is the usual skew-group algebra (see also \cite{IT}). Thus, in this case, if $d \geq 3$, we deduce from Theorem \ref{2.5} that
 $$\frac{{}^{\bot_{d-3}} S}{[S]} \simeq \mod \left(\uEnd_R(S)\right)^{\op} = \mod \left(\frac{S * G}{(e)}\right)^{\op}$$
  where the idempotent
 $$e = \frac{\sum_{g \in G} g}{\# G}$$
 corresponds to the trivial representation of $G$.

 Remark that, if $R$ is not Gorenstein, then the exact category of Cohen-Macaulay modules $\CM(R)$ is not Frobenius. Moreover $R$ is Gorenstein if and only if $G \in \SL_n(\CC)$ \cite[Theorem 1]{Wata}. In the Gorenstein case, see also \cite{AIR} which realizes $\underline\CM(R)$ as a generalized cluster category.

 For example, if $G = \langle \zeta \operatorname{Id} \rangle$ where $\zeta$ is a primitive $n$-th root of the unity, then
 $$\frac{{}^{\bot_{d-3}} S}{[S]} \simeq \mod \CC Q / I $$
 where
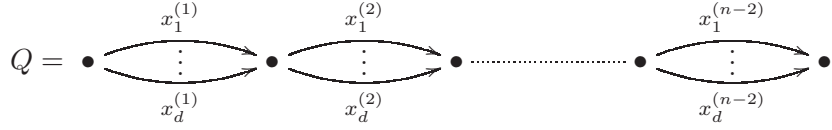
 $$Q = \boxinminipage{\xymatrix@C=2cm{\bullet \ar@/^.28cm/[r]^{x_1^{(1)}} \ar@/_.28cm/[r]^{\vdots}_{x_d^{(1)}} & \bullet \ar@/^.28cm/[r]^{x_1^{(2)}} \ar@/_.28cm/[r]^{\vdots}_{x_d^{(2)}} & \bullet \ar@{..}[r] & \bullet\ar@/^.28cm/[r]^{x_1^{(n-2)}}  \ar@/_.28cm/[r]^{\vdots}_{x_d^{(n-2)}} & \bullet }}$$
 and $I$ is the ideal of $\CC Q$ generated by the relations $x_{i}^{(k)} x_j^{(k+1)} = x_{j}^{(k)} x_i^{(k+1)}$ for $1 \leq i < j \leq d$ and $1 \leq k \leq n-3$ (see \cite{Aus1, IT, Yoshi} for details about the computation).
\end{exm}

\def\newboxedcommand#1#2 
{%
 \def\newboxedcommandlocala##1##2.{##2}%
 \edef\newboxedcommandlocalb{\expandafter\newboxedcommandlocala\string#1.}%
 \expandafter\newsavebox\csname\newboxedcommandlocalb savebox\endcsname%
 \expandafter\sbox\csname\newboxedcommandlocalb savebox\endcsname{#2}%
 \expandafter\newlength\csname\newboxedcommandlocalb largeurbox\endcsname%
 \expandafter\settowidth\csname\newboxedcommandlocalb largeurbox\endcsname{\usebox{\csname\newboxedcommandlocalb savebox\endcsname}}%
 \edef#1{\noexpand\begin{minipage}{\csname\newboxedcommandlocalb largeurbox\endcsname}\usebox{\csname\newboxedcommandlocalb savebox\endcsname}\noexpand\end{minipage}}%
}

\newboxedcommand\Sone
{
\small\objectmargin={0.5pt}$\xymatrix@C=.2cm@R=.2cm{
1
}$
}

\newboxedcommand\Stwo
{
\small\objectmargin={0.5pt}$\xymatrix@C=.2cm@R=.2cm{
2
}$
}

\newboxedcommand\Sthree
{
\small\objectmargin={0.5pt}$\xymatrix@C=.2cm@R=.2cm{
3
}$
}

\newboxedcommand\Sfour
{
\small\objectmargin={0.5pt}$\xymatrix@C=.2cm@R=.2cm{
4
}$
}

\newboxedcommand\Sfive
{
\small\objectmargin={0.5pt}$\xymatrix@C=.2cm@R=.2cm{
5
}$
}

\newboxedcommand\Ssix
{
\small\objectmargin={0.5pt}$\xymatrix@C=.2cm@R=.2cm{
6
}$
}

\newboxedcommand\Pone
{
\small\objectmargin={0.5pt}$\xymatrix@C=.2cm@R=.2cm{
1 \ar[dr] & & \\
& 2 \ar[dr] & \\
& & 3
}$
}

\newboxedcommand\Ptwo
{
\small\objectmargin={0.5pt}$\xymatrix@C=.2cm@R=.2cm{
& 2 \ar[dl] \ar[dr] & \\
3 \ar[dr] & & 4 \ar[dl] \\
& 5 & \\
}$
}

\newboxedcommand\Pthree
{
\small\objectmargin={0.5pt}$\xymatrix@C=.2cm@R=.2cm{
3 \ar[dr] & & \\
& 5 \ar[dr] & \\
& & 6
}$
}

\newboxedcommand\Pfour
{
\small\objectmargin={0.5pt}$\xymatrix@C=.2cm@R=.2cm{
4 \ar[dr] & \\
& 5 \\
}$
}

\newboxedcommand\Pfive
{
\small\objectmargin={0.5pt}$\xymatrix@C=.2cm@R=.2cm{
5 \ar[dr] & \\
& 6  \\
}$
}

\newcommand\Psix\Ssix

\newcommand\Ione\Sone

\newboxedcommand\Itwo
{
\small\objectmargin={0.5pt}$\xymatrix@C=.2cm@R=.2cm{
1 \ar[dr] & \\
& 2 \\
}$
}

\newcommand\Ithree\Pone

\newboxedcommand\Ifour
{
\small\objectmargin={0.5pt}$\xymatrix@C=.2cm@R=.2cm{
2 \ar[dr] & \\
& 4 \\
}$
}

\newcommand\Ifive\Ptwo

\newcommand\Isix\Pthree

\newboxedcommand\threefive
{
\small\objectmargin={0.5pt}$\xymatrix@C=.2cm@R=.2cm{
3 \ar[dr] & \\
& 5 \\
}$
}

\newboxedcommand\threefourfive
{
\small\objectmargin={0.5pt}$\xymatrix@C=.2cm@R=.2cm{
3 \ar[dr] &  & 4 \ar[dl] \\
& 5 & \\
}$
}

\newboxedcommand\twothreefour
{
\small\objectmargin={0.5pt}$\xymatrix@C=.2cm@R=.2cm{
& 2 \ar[dl] \ar[dr] & \\
3 & & 4 \\
}$
}

\newboxedcommand\twothree
{
\small\objectmargin={0.5pt}$\xymatrix@C=.2cm@R=.2cm{
2 \ar[dr] & \\
& 3 \\
}$
}

Recently Herschend, Iyama and Oppermann studied \emph{$n$-representation-finite algebras} (see for example \cite{HI1, HI2, 5, IO}), which are finite dimensional algebras with global dimension at most $n$ and has an $n$-cluster tilting module. We can apply our results to all of those algebras. From now on, we describe explicitely a simple example taken from \cite{5}.


\begin{exm}\label{4.1.}
Let $\Lambda$ be the Auslander algebra of $k \vec{A}_3$. That is $kQ/R$ where $Q$ is the following quiver
$$
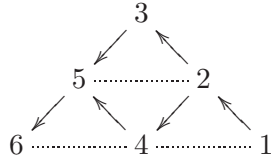
\xymatrix@C=0.4cm@R0.4cm{
&&3 \ar[dl]\\
&5 \ar[dl] \ar@{.}[rr] &&2 \ar[dl] \ar[ul]\\
6 \ar@{.}[rr] &&4 \ar[ul] \ar@{.}[rr] &&1 \ar[ul]}$$
and the ideal of relations $R$ is generated by the mesh relations symbolized by dashed lines.
Then, using the method introduced in \cite[{\S 1}]{5}, one can compute a cluster tilting subcategory $\M$ of $\mod \Lambda$. It is generated by the direct sum of the $\tau_2^i(D\Lambda)$ for $i \geq 0$ where $\tau_2$ is the $2$-AR translation.
\end{exm}

Let us compute the $\tau_2^i(D\Lambda)$, each indecomposable module being depicted by its compositions series representation:


\begin{eqnarray*}
&D\Lambda=
\begin{smallmatrix}
1
\end{smallmatrix}\oplus
\begin{smallmatrix}
1&\\
&2
\end{smallmatrix}\oplus
\begin{smallmatrix}
1&&\\
&2&\\
&&3
\end{smallmatrix}\oplus
\begin{smallmatrix}
&2\\
4&
\end{smallmatrix}\oplus
\begin{smallmatrix}
&2&\\
4&&3\\
&5&
\end{smallmatrix}\oplus
\begin{smallmatrix}
&&3\\
&5&\\
6&&
\end{smallmatrix},&
\end{eqnarray*}


\begin{eqnarray*}
\tau_2(D\Lambda)=
\begin{smallmatrix}
4
\end{smallmatrix}\oplus
\begin{smallmatrix}
4&\\
&5
\end{smallmatrix}\oplus
\begin{smallmatrix}
&5\\
6&
\end{smallmatrix} \quad \quad
\tau_2^2(D\Lambda)=
\begin{smallmatrix}
6
\end{smallmatrix} \quad \quad
\tau_2^3(D\Lambda)=0.&
\end{eqnarray*}

The quiver of $\M$ is given in Figure \ref{QM}.

\begin{figure}
$$\xymatrix@C=0.0cm@R0.0cm{
&&&&{\begin{smallmatrix}
1&&\\
&2&\\
&&3
\end{smallmatrix}}\ar[drr]\\
&&{\begin{smallmatrix}
&2&\\
4&&3\\
&5&
\end{smallmatrix}}\ar[drr]\ar[urr]
&&&&{\begin{smallmatrix}
1&&\ \\
&2&
\end{smallmatrix}}\ar[drrrr]\ar@{.}[ddlll]\\
{\begin{smallmatrix}
&&3\\
&5&\\
6&&
\end{smallmatrix}}\ar[urr]
&&&&{\begin{smallmatrix}
&2&\ \\
4&&
\end{smallmatrix}}\ar[urr]\ar@{.}[ddlll]
&&&&&&{\begin{smallmatrix}
\ &1&\
\end{smallmatrix}}\ar@{.}[ddlllll]\\
&&&{\begin{smallmatrix}
4&&\ \\
&5&
\end{smallmatrix}}\ar[uul]\ar[drr]\\
&{\begin{smallmatrix}
\ \\
&5&\ \\
6&&
\end{smallmatrix}}\ar[uul]\ar[urr]
&&&&{\begin{smallmatrix}
\ &4&\
\end{smallmatrix}}\ar[uul]\ar@{.}[dddlll]\\
\ \\
\ \\
&&{\begin{smallmatrix}
\ \\
\ &6&\ \\
\
\end{smallmatrix}}\ar[uuul]}$$
\caption{Quiver of $\M$}
\label{QM}
\end{figure}

Then we can calculate $\Omegab \M$ easily since in this case $\Omegab \M=\oM^\bot=\{ X\in \mod\Lambda \text{ }| \text{ } \overline\Hom_{\Lambda}(\M,X)=0 \}$. We give a full view of these categories in Figure 2.

In this example, the quiver of $\mod \Lambda/[\M]$ is the following.

$$
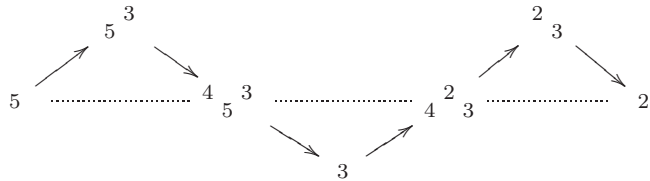
\xymatrix@C=0.4cm@R0.4cm{
&{\begin{smallmatrix}
&&3\ \\
&5&
\end{smallmatrix}}\ar[dr]
&&&&{\begin{smallmatrix}
2&\ \\
&3
\end{smallmatrix}}\ar[dr]\\
{\begin{smallmatrix}
\ &5&\
\end{smallmatrix}} \ar[ur] \ar@{.}[rr]
&&{\begin{smallmatrix}
4&&3\ \\
&5&
\end{smallmatrix}}\ar[dr] \ar@{.}[rr]
&&{\begin{smallmatrix}
&2&\ \\
4&&3
\end{smallmatrix}}\ar[ur] \ar@{.}[rr]
&&{\begin{smallmatrix}
\ &2&\
\end{smallmatrix}}\\
&&&{\begin{smallmatrix}
\ &3&\
\end{smallmatrix}} \ar[ur]}$$
The quiver of $\uM$ is the following.

$$
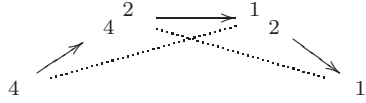
\xymatrix@C=0.4cm@R0.4cm{
&{\begin{smallmatrix}
&&2\ \\
&4&
\end{smallmatrix}}\ar[rr] \ar@{.}[drrr]
&&{\begin{smallmatrix}
1&\ \\
&2
\end{smallmatrix}}\ar[dr]\\
{\begin{smallmatrix}
\ &4&\
\end{smallmatrix}} \ar[ur] \ar@{.}[urrr]
&&&&{\begin{smallmatrix}
\ &1&\
\end{smallmatrix}}}$$
As expected, we obtain that $\mod \Lambda/[\M] \simeq \mod\uM$. One can also calculate and check the equivalence $\mod \Lambda/[\oM^\bot] \simeq \mod\oM$.

\begin{exm}\label{4.2}
Keeping previous notation, let
$$\M'=\add
\left(\begin{smallmatrix}
6
\end{smallmatrix}\oplus
\begin{smallmatrix}
&&5\\
&6&
\end{smallmatrix}\oplus
\begin{smallmatrix}
&&3\\
&5&\\
6&&
\end{smallmatrix}\oplus
\begin{smallmatrix}
4&&\\
&5&
\end{smallmatrix}\oplus
\begin{smallmatrix}
&2&\\
4&&3\\
&5&
\end{smallmatrix}\oplus
\begin{smallmatrix}
1&&\\
&2&\\
&&3
\end{smallmatrix}\oplus
\begin{smallmatrix}
4
\end{smallmatrix}\oplus
\begin{smallmatrix}
&&2\\
&4&
\end{smallmatrix}\right).
$$
We get ${\mathcal P}\subseteq \M'\subseteq \M$. Hence $\M'$ is a rigid subcategory of $\mod \Lambda$ containing all the projective objects. By Theorem \ref{2.5}, we get $\M'_L/[\M']\simeq\mod\underline{\M'}$. In fact
$$\M'_L=
\M'\oplus
\add\left(\begin{smallmatrix}
5
\end{smallmatrix}\oplus
\begin{smallmatrix}
&3\\
5&
\end{smallmatrix}\oplus
\begin{smallmatrix}
4&&3\\
&5&
\end{smallmatrix}\right)$$
and the quiver of $\M'_L/[\M']$ is the following:
$$
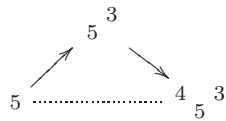
\xymatrix@C=0.4cm@R0.4cm{
&{\begin{smallmatrix}
&&3 \\
&5&
\end{smallmatrix}} \ar[dr]\\
{\begin{smallmatrix}
5
\end{smallmatrix}} \ar[ur] \ar@{.}[rr] &&{\begin{smallmatrix}
4&&3\\
&5&
\end{smallmatrix}}
}$$
which is isomorphic to the AR quiver of
$\uM'\simeq k[{\begin{smallmatrix}
4
\end{smallmatrix}} \longrightarrow {\begin{smallmatrix}
&2\\
4&
\end{smallmatrix}}]$.
\end{exm}

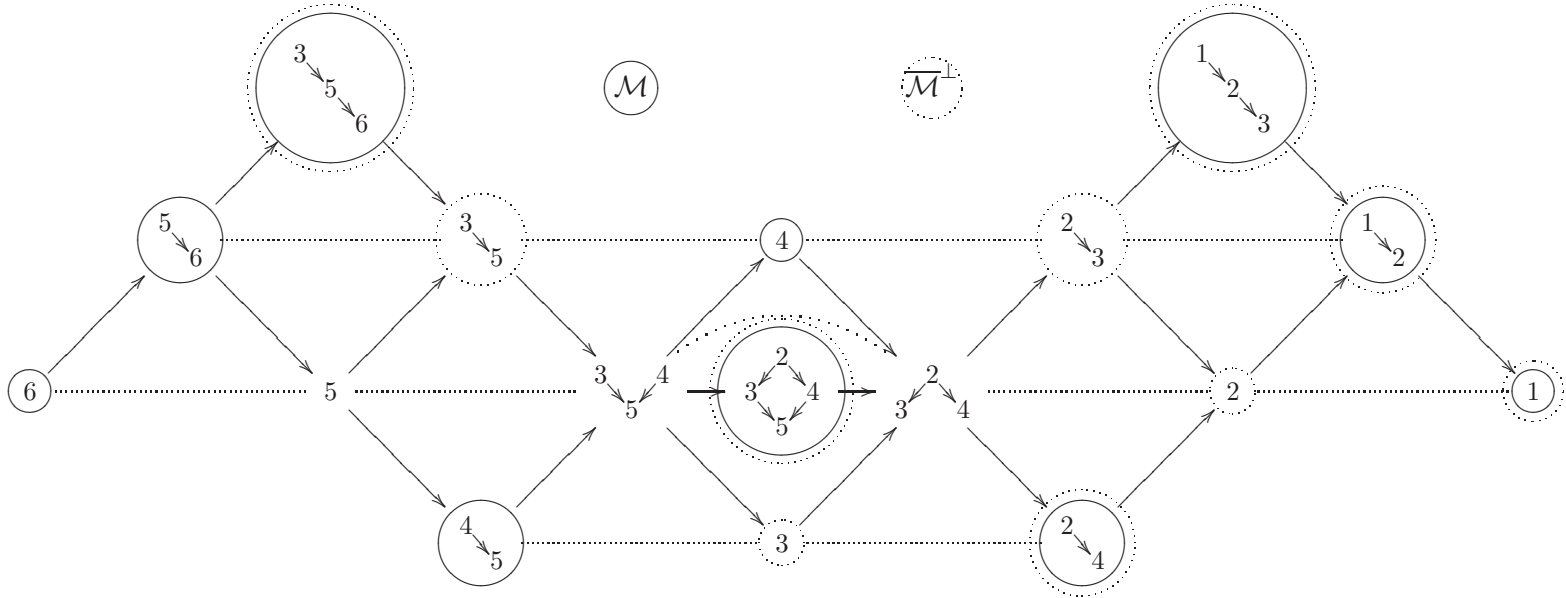
\begin{sidewaysfigure}

\vspace*{10cm}

$$\xymatrix@!C=1.15cm@!R=1.15cm{
 & & \Pthree \POS*\frm<1cm>{o} \POS*\frm<1.1cm>{.o} \ar[dr] & & \mathcal{M} \POS*\frm<.4cm>{o} & & \overline{\mathcal{M}}^\perp \POS*\frm<.4cm>{.o} & & \Pone \POS*\frm<1cm>{o} \POS*\frm<1.1cm>{.o} \ar[dr] & & \\
 & \Pfive \POS*\frm<.6cm>{o} \ar[ur] \ar[dr] \ar@{..}[rr] & & \threefive \POS*\frm<.6cm>{.o} \ar[dr] \ar@{..}[rr] & & \Sfour \POS*\frm<.3cm>{o} \ar[dr] \ar@{..}[rr] & & \twothree \POS*\frm<.6cm>{.o} \ar[ur] \ar[dr] \ar@{..}[rr] & & \Itwo \POS*\frm<.6cm>{o} \POS*\frm<.7cm>{.o} \ar[dr] & \\
 \Ssix \POS*\frm<.3cm>{o} \ar[ur] \ar@{..}[rr] & & \Sfive \ar[ur] \ar[dr] \ar@{..}[rr] & & \threefourfive \ar[ur] \ar[dr] \ar[r] \ar@/^1cm/@{..}[rr] & \Ptwo \POS*\frm<.85cm>{o} \POS*\frm<.95cm>{.o} \ar[r] & \twothreefour \ar[ur] \ar[dr] \ar@{..}[rr] & & \Stwo \POS*\frm<.3cm>{.o} \ar[ur] \ar@{..}[rr] & & \Sone \POS*\frm<.3cm>{o} \POS*\frm<.4cm>{.o} \\
 & & & \Pfour \POS*\frm<.6cm>{o} \ar[ur] \ar@{..}[rr] & & \Sthree \POS*\frm<.3cm>{.o} \ar[ur] \ar@{..}[rr] & & \Ifour \POS*\frm<.6cm>{o} \POS*\frm<.7cm>{.o} \ar[ur] & & &
}$$
\caption{AR-quiver of the Auslander algebra of $kA_3$}
\end{sidewaysfigure}

\section*{Acknowledgments}
The authors would like to thank Osamu Iyama for his helpful advises and corrections and the referee for his valuable comments.

\end{document}